\documentclass{amsart}
\usepackage[T1]{fontenc}
\usepackage{amsmath, amssymb, amsthm, amsfonts}
\usepackage{mathrsfs}

\usepackage{lmodern}
\usepackage{fontawesome5}

\usepackage{graphicx}
\usepackage[dvipsnames]{xcolor}
\usepackage{mdframed}
\usepackage{tikz-cd}
\usepackage{booktabs}
\usepackage{makecell}
\usepackage{comment}
\usepackage{siunitx}
\usepackage{pgfplots}
\pgfplotsset{width=10cm,compat=1.18}
\usepackage{subcaption}

\usepgfplotslibrary{colormaps}
\pgfplotsset{
    colormap={blueorange}{rgb255=(0,114,178) rgb255=(86,180,233) rgb255=(240,228,66) rgb255=(230,159,0) rgb255=(213,94,0)}, 
    colormap={blueorange-mid}{rgb255=(213,94,0) rgb255=(230,159,0) rgb255=(240,228,66) rgb255=(86,180,233) rgb255=(0,114,178) rgb255=(0,114,178) rgb255=(86,180,233) rgb255=(240,228,66) rgb255=(230,159,0) rgb255=(213,94,0)}, 
    colormap={onlyblue}{rgb255=(0,114,178) rgb255=(86,180,233)},
    colormap={onlyorange}{rgb255=(213,94,0) rgb255=(230,159,0)},
}

\usepackage{fancyhdr}
\usepackage[hang,flushmargin]{footmisc}

\usepackage{enumitem}
\usepackage{mathtools}

\usepackage[alphabetic, initials, msc-links]{amsrefs}
\makeatletter
\renewcommand{\BibLabel}{%
    \Hy@raisedlink{\hyper@anchorstart{cite.\CurrentBib}\hyper@anchorend}%
    [\thebib]%
}
\makeatother
\makeatletter
    \renewcommand{\MR}[1]{\ (\href{https://mathscinet.ams.org/mathscinet-getitem?mr=MR#1}{MR#1})}
\makeatother

\usepackage{hyperref}
\hypersetup{
    colorlinks=true,
    linkcolor=-Cerulean,
    filecolor=magenta,
    urlcolor=Cerulean,
    citecolor=Cerulean
}

\usepackage[capitalize, nameinlink]{cleveref}
\crefname{equation}{}{}
\crefname{prop}{Proposition}{Propositions}
\crefname{figure}{Figure}{Figures}
\crefname{conjecture}{Conjecture}{Conjectures}
\crefname{question}{Question}{Questions}

\graphicspath{ {./Images/} }

\usepackage{soul}
\usepackage{lineno}

\definecolor{electricblue}{RGB}{125, 249, 255}
\definecolor{darkorange}{RGB}{213,94,0}
\definecolor{darkblue}{RGB}{0,114,178}
\definecolor{outlineorange}{RGB}{170,75,0}
\definecolor{outlineblue}{RGB}{0,91,142}

\theoremstyle{definition}
\newtheorem{theorem}{Theorem}[section]

\newtheorem{conjecture}[theorem]{Conjecture}
\newtheorem{definition}[theorem]{Definition}

\newtheorem{proposition}[theorem]{Proposition}
\newtheorem{corollary}[theorem]{Corollary}

\newtheorem{algorithm}[theorem]{Algorithm}

\newtheorem{remark}[theorem]{Remark}
\newtheorem{question}[theorem]{Question}

\newtheorem{letterconj}{Conjecture}
\newtheorem{letterques}[letterconj]{Question}

\DeclareMathOperator{\Spec}{\mathrm{Spec}}

\DeclareMathOperator{\PGL}{PGL}

\title[(Ir)reducibility of $p$-rank Strata]{Heuristics for (ir)reducibility of $p$-rank strata of the Moduli Space of Hyperelliptic Curves}


\author[T.\ Bouchet]{Thomas Bouchet}
\address{Thomas Bouchet | Laboratoire J.A.\ Dieudonn\'e, Universit\'e C\^ote d'Azur, France}
\email{\href{mailto:thomas.bouchet@univ-cotedazur.fr}{\tt thomas.bouchet@univ-cotedazur.fr}}
\urladdr{\href{https://sites.google.com/view/thomas-bouchet/accueil}{\tt https://sites.google.com/view/thomas-bouchet/accueil}}

\author[E.\ Davis]{Erik Davis}
\address{Erik Davis | Department of Mathematics, Texas A\&M University, USA}
\email{\href{mailto:davis7e@tamu.edu}{\tt davis7e@tamu.edu}}
\urladdr{\href{https://davis7e.wordpress.com}{\tt https://davis7e.wordpress.com}}

\author[S.R.\ Groen]{Steven R.\ Groen}
\address{Steven R.\ Groen | Department of Mathematics, Lehigh University, USA}
\email{\href{mailto:stevengroen95@gmail.com}{\tt stevengroen95@gmail.com}}
\urladdr{\href{https://sites.google.com/view/stevengroen}{\tt https://sites.google.com/view/stevengroen}}

\author[Z.\ Porat]{Zachary Porat}
\address{Zachary Porat | Department of Mathematics and Computer Science, Wesleyan University, USA}
\email{\href{mailto:zporat@wesleyan.edu}{\tt zporat@wesleyan.edu}}
\urladdr{\href{https://zporat.github.io}{\tt https://zporat.github.io}}

\author[B.\ York]{Benjamin York}
\address{Benjamin York | Department of Mathematics, University of Connecticut, USA}
\email{\href{mailto:benjamin.york@uconn.edu}{\tt benjamin.york@uconn.edu}}
\urladdr{\href{https://benjamin-york.github.io/}{\tt https://benjamin-york.github.io}}

\begin{document}

\begin{abstract}
     Let $\mathcal{H}_g$ denote the moduli space of smooth hyperelliptic curves of genus $g$ in characteristic $p\geq 3$, and let $\mathcal{H}_g^f$ denote the $p$-rank $f$ stratum of $\mathcal{H}_g$ for $0 \leq f \leq g$.  Achter and Pries note in \cite{AP11} that determining the number of irreducible components of $\mathcal{H}_g^f$ would lead to several intriguing corollaries.  In this paper, we present a computational approach for estimating the number of irreducible components in various $p$-rank strata.  Our strategy involves sampling curves over finite fields and calculating their $p$-ranks. From the data gathered, we conjecture that the non-ordinary locus is geometrically irreducible for all genera $g> 1$.  The data also leads us to conjecture that the moduli space $\mathcal{H}^{g-2}_g$ is irreducible and suggests that $\mathcal{H}^f_g$ is irreducible for all $1 \leq f \leq g$.  We conclude with a brief discussion on $\mathcal{H}^0_g$.
\end{abstract}

\maketitle

\section{Introduction}

Throughout this paper, we consider hyperelliptic curves defined over a perfect field $k$ of characteristic $p\geq 3$. Given a smooth curve $C / k$ of genus $g_C$, we define the \textbf{$p$-rank} of $C$ to be the integer $f_C$ such that $\mathrm{Jac}(C)[p](\bar{k}) \cong (\mathbb{Z}/p\mathbb{Z})^{f_C}$. By \cite{Oda69}, the $p$-rank of $C$ equals the semisimple rank of the Frobenius operator on $\mathrm{H}^1(C,\mathcal{O}_C)$, i.e.\ the rank of the $g\textsuperscript{th}$ semilinear power.  In particular, we have $0 \leq f_C \leq g_C$.

Given a prime number $p \geq 3$, let $\mathcal{H}_g$ be the moduli space of smooth hyperelliptic curves over $\overline{\mathbb{F}}_p$ of genus $g$ up to geometric isomorphism.  In addition, the moduli space $\mathcal{H}_g$ can be defined over $\mathbb{F}_p$. For any $0 \leq f \leq g$, let $\mathcal{H}_g^f$ denote the \textbf{$p$-rank $f$ stratum} of $\mathcal{H}_g$, i.e.\ the moduli space of smooth hyperelliptic curves of genus $g$ and $p$-rank at most $f$. The curves whose $p$-rank is less than $g$ form the \textbf{non-ordinary locus} $\mathcal{H}_g^{g-1}$. In this paper, we are interested in the geometry of $\mathcal{H}_g^f$, in particular its number of irreducible components. 

Let $\overline{\mathcal{H}}_g^f$ denote the Deligne-Mumford compactification of $\mathcal{H}_g^f$. The full moduli space $\overline{\mathcal{H}}_g = \overline{\mathcal{H}}_g^g$ is known to be irreducible of dimension $2g-1$. Moreover, \cite{GP05}*{Proposition 2.1} shows that the compactified $p$-rank strata $\overline{\mathcal{H}}_g^f$ are pure of codimension $g-f$ and \cite{AP11}*{Lemma 3.2} explains that every component of $\overline{\mathcal{H}}_g^f$ intersects $\mathcal{H}_g^f$. In particular, for any $0 \leq f \leq g$ the space $\overline{\mathcal{H}}_g^f$ is nonempty of dimension $g+f-1$ and the generic point of any component corresponds to a smooth hyperelliptic curve of genus $g$ with $p$-rank $f$. Little is known about the number of geometrically irreducible components of $\overline{\mathcal{H}}_g^f$ for $f<g$ in general.

Let $\mathcal{A}_g^f$ be the moduli space of principally polarized abelian varieties of dimension $g$ and $p$-rank at most $f$, such that the Torelli morphism injects $\mathcal{H}_g^f$ into $\mathcal{A}_g^f$.  By \cite{Oor01}*{Corollary 1.5}, we know that the non-ordinary locus $\mathcal{A}_g^{g-1}$ is geometrically irreducible whenever $g>1$. Furthermore, by \cite{Cha05}*{Remark 4.7}, we know that $\mathcal{A}_g^f$ is geometrically irreducible for all $g \geq 3$.  However, much less is known about the $p$-rank strata of $\mathcal{H}_g$.

In the $g=1$ case, we have $\mathcal{H}_1=\mathcal{M}_{1,1} \cong \mathcal{A}_1$ via the Torelli morphism. The $p$-rank $0$ stratum $\mathcal{H}_1^0$ has dimension zero and the Deuring-Eichler mass formula (see \cite{Sil09}*{Theorem 4.1(c)}) gives the exact number of points of $\mathcal{H}_1^0$, which is roughly $\frac{p-1}{12}$. The $g=2$ case is also well understood.  \cite{Oor01}*{Corollary 1.5} provides that $\mathcal{A}_2^1$ is geometrically irreducible, and the geometric irreducibility of $\mathcal{H}_2^1$ follows from the fact that the generic point of $\mathcal{A}_2^1$ corresponds to the Jacobian of a smooth hyperelliptic curve. Finally, the number of geometric components of $\mathcal{H}_2^0$, given in \cite{IKO86}*{Remarks 2.16 and 2.17}, is shown to go to infinity as $p$ increases.

For $g \geq 3$ and $f<g$, the number of geometric components of $\mathcal{H}_g^f$ is less accessible, since it is no longer true that the generic point of $\mathcal{A}_g^f$ corresponds to the Jacobian of a hyperelliptic curve.  \cite{AP11}*{Section 3.7} observes that an understanding of the number of components of $\overline{\mathcal{H}}_g^f$ has interesting applications. More precisely, if $\overline{\mathcal{H}}_g^f$ has only one component, then in particular every component contains the moduli point of a chain of elliptic curves clutched together. As a corollary, the closure of any component of $\mathcal{H}_g^f$ would intersect the space $\Delta_i[\mathcal{H}_g]^f$, consisting of singular curves that are obtained by clutching at a ramified point two hyperelliptic curves whose respective genera are $i$ and $g-i$ and whose $p$-ranks add up to $f$. 

In this paper, we employ a computational approach to investigate this elusive question.  More precisely, we gather data by sampling curves over finite fields and computing their $p$-ranks. Using the Lang-Weil bounds described in \cite{LW54}, our process results in strong evidence for the number of components of $\mathcal{H}_g^f$. 

Let $c_{g,f,p,r}$ be the number of components of $\mathcal{H}_g^f \otimes \textrm{Spec}(\mathbb{F}_{p^r})$ and let $c_{g,f,p}$ be the number of components of $\mathcal{H}_g^f \otimes \text{Spec}(\overline{\mathbb{F}}_p)$. Recall that $\overline{\mathcal{H}}_g^f$ is pure of codimension $g-f$ and $\mathcal{H}_g^f$ is dense in $\overline{\mathcal{H}}_g^f$.
Thus, if we take a sample $S$ of $\overline{\mathbb{F}}_p$-isomorphism classes of hyperelliptic curves of genus $g$ over $\mathbb{F}_{p^r}$, then the Lang-Weil bounds prescribe the heuristic
\begin{equation} \label{eq:L-W intro}
\frac{ \# \{[C] \in S \; | \; f_C\leq f\}}{ \# S }  \approx \frac{c_{g,f,p,r}}{p^{r(g-f)}}.
\end{equation}
Here $[C]$ is the $\overline{\mathbb{F}}_p$-isomorphism class of a curve $C/\mathbb{F}_{p^r}$.  This allows us to estimate the number of components $c_{g,f,p,r}$, which in turn provides evidence for the number of geometric components $c_{g,f,p}$ when $r$ is varied.  See \cref{Sec:LangWeil} for a complete discussion of how we derive this heuristic.

In order to gather data on $c_{g,f,p,r}$, we use two different methods for generating isomorphism classes of hyperelliptic curves. One method is based on a family of pairwise non-isomorphic hyperelliptic curves over $\mathbb{F}_q$, which is relatively efficient compared to methods based on the invariants of the curves. The second method is based on an adaptation of \cite{How25}, in which hyperelliptic curves over finite fields are enumerated. To make this second method more computationally viable, we assume that all branch points are defined over $\mathbb{F}_{p^r}$. The $p$-rank of a curve $C$ is computed efficiently as the semisimple rank of a Hasse-Witt matrix of $C$, that is a matrix representing the semilinear action of Frobenius on $\mathrm{H}^1(C, \mathcal{O}_C)$ (see \cref{Sec:HasseWitt}).

Since the sample size is limited by computational capacity, the estimate from \cref{eq:L-W intro} is most precise for the non-ordinary locus $\mathcal{H}_g^{g-1}$. We use a large range of characteristics $p$, vary the exponent $1 \leq r \leq 4$, and run computations for genera $g$ with $3 \leq g \leq 20$, which leads to the following conjecture, extending the known fact that $\mathcal{H}_2^1$ is geometrically irreducible.

\renewcommand{\theletterconj}{\Alph{letterconj}}
\renewcommand{\theletterques}{\Alph{letterques}}
\begin{letterconj}[\cref{conj: non-ordinary}]
    For $g \geq 3$ and $p \geq 3$, the non-ordinary locus $\mathcal{H}_g^{g-1}$ is geometrically irreducible.
\end{letterconj}

Our methodology allows us to determine heuristics for lower $p$-rank strata. Since curves with smaller $p$-rank are rarer, the evidence is generally less strong than for the non-ordinary locus. Nevertheless, the data allows us to make the following conjecture for the moduli space $\mathcal{H}_g^{g-2}$.

\begin{letterconj}[\cref{conj: f=g-2}]
For $g \geq 3$ and $p \geq 3$, the moduli space $\mathcal{H}_g^{g-2}$ is geometrically irreducible. 
\end{letterconj}

Additionally, our data suggests that in general, this phenomenon remains the same for other $p$-rank strata.  That is, the moduli space $\mathcal{H}_g^f$ is geometrically irreducible for $1 \leq f \leq g$ and $p \geq 3$.  While the data is leading in this direction, we have insufficient evidence to make a concrete conjecture.  Instead, we pose an open question to the community. 

\begin{letterques}[\cref{q:1<=f<=g}]
For $1 \leq f \leq g$ and $p \geq 3$, is $\mathcal{H}_g^f$ always geometrically irreducible? \end{letterques}

In contrast, our data on the $p$-rank $0$ locus $\mathcal{H}_3^0$ suggests a different trend, namely that the number of components of $\mathcal{H}_3^0$ grows with $p$. This leads us to make a conjecture.

\begin{letterconj}[\cref{conj:g=3 f=0}]
The number of components of $\mathcal{H}_3^0$ goes to infinity as $p$ grows.
\end{letterconj}

\noindent Furthermore, since the observed behavior for $\mathcal{H}_3^0$ extends the established behavior of $\mathcal{H}_1^0$ and $\mathcal{H}_2^0$, we pose the following question.

\begin{letterques}[\cref{q:f=0}]
For $g \geq 1$, does the number of geometric components of $\mathcal{H}_g^0$ go to infinity as $p$ grows?
\end{letterques}

Should \cref{conj: non-ordinary,conj: f=g-2,conj:g=3 f=0} hold and \cref{q:1<=f<=g,q:f=0} be answered in the affirmative, we would have a complete understanding of the number of components of $\mathcal{H}_g^f$ for any $g$ and $f$, generalizing the results in the literature for the specific cases when $g \leq 2$ or when $f=g$.

The overall structure of the paper is as follows.  We start in \cref{sec: background} with a theoretical background on hyperelliptic curves, Hasse-Witt matrices and Lang-Weil bounds. In \cref{Sec:CurveGen}, the two methods for generating pairwise non-isomorphic hyperelliptic curves over $\mathbb{F}_q$ are outlined. Then, we turn our attention to analyzing data for the non-ordinary locus $\mathcal{H}_g^{g-1}$ in \cref{sec:nonord}, presenting strong evidence that this locus is irreducible for genera $g$ with $3 \leq g \leq 20$ and $p \geq 3$. We conclude in \cref{sec:f<g-1} with a discussion of heuristics for other $p$-rank strata of $\mathcal{H}_g$.

Our work makes extensive use of the computer algebra system \textsc{Magma} \cite{magma}. The scripts written for and data collected throughout this project can be found in the GitHub repository \cite{Code}.

\section*{Acknowledgments}
This project originated at the 2024 Arizona Winter School on Abelian Varieties. The authors would like to thank the organizers of the Arizona Winter School for providing the conditions for this project to begin.  In particular, the authors are deeply grateful to Rachel Pries for suggesting this line of inquiry. We further thank Shiruo Wang for his contributions to the early stages of this project.

\section{Background} \label{sec: background}

\subsection{Hyperelliptic Curves over $\mathbb{F}_q$} \label{Sec:prank}

In this section, we will let $p$ denote an odd prime, $r \geq 1$ be an integer, and $q = p^r$.  For our purposes, a curve is a geometrically connected algebraic variety of dimension one.  In particular, we are interested in hyperelliptic curves, which we define as follows.

\begin{definition}
    Let $k$ be a field not of characteristic $2$ and $C/k$ a projective curve. We say that $C$ is \textbf{hyperelliptic} if it is a double cover of $\mathbb{P}^1$.
\end{definition}

Throughout this paper, we work exclusively with smooth and irreducible hyperelliptic curves. Let $C/\mathbb{F}_q$ be a hyperelliptic curve defined over $\mathbb{F}_q$. Let $g_C$ denote the genus of $C$. Then, $C$ has an affine model of the form $y^2 = f(x)$ where $f(x) \in \mathbb{F}_q[x]$ is a separable polynomial of degree $2g_C + 1$ or $2g_C + 2$. 

\begin{proposition}[\cite{LR12}*{Prop.\ 1.10}]\label{prop:iso_hyp_cur}
    Let $C_1$ and $C_2$ be two hyperelliptic curves defined over $k$, where $k$ is a field not of characteristic $2$.  Furthermore, let $y^2 = f_i(x)$ be an affine model of $C_i$. We assume that $C_1$ and $C_2$ are $k$-isomorphic. Then, there exist $M = \begin{pmatrix}a & b \\ c & d\end{pmatrix} \in \mathrm{GL}_2(k)$ and $e\in k^*$ such that the isomorphism is given by \[(x,y) \mapsto \left(\frac{ax+b}{cx+d},\frac{ey}{(cx+d)^{g_C+1}}\right).\] 
\end{proposition}

Given a hyperelliptic curve $C/\mathbb{F}_q$ with genus $g_C$, we let $\mathrm{Jac}(C)$ denote the Jacobian of $C$, an abelian variety of dimension $g_C$. A major invariant of interest is the \textbf{$p$-rank} of $C$.  The $p$-rank of $C$, denoted $f_C$, is the integer $0 \leq f_C \leq g_C$ such that $\#\mathrm{Jac}(C)[p](\overline{\mathbb{F}}_{p}) = p^{f_C}$, where $\mathrm{Jac}(C)[p]$ is the $p$-torsion group scheme of $\mathrm{Jac}(C)$.  We say that a curve with $p$-rank $g_C$ is \textbf{ordinary}, and a curve with $p$-rank less than $g_C$ is \textbf{non-ordinary}.

Fix an integer $g \geq 2$ and a prime $p \geq 3$. We denote the moduli space of smooth hyperelliptic curves of genus $g$ over $\overline{\mathbb{F}}_p$ up to isomorphism by $\mathcal{H}_{g}$, which is known to be irreducible for any choice of $g$. For $0 \leq f \leq g$, we denote by $\mathcal{H}_{g}^{f}$ the moduli space of smooth hyperellipitic curves of genus $g$ and $p$-rank at most $f$ over $\overline{\mathbb{F}}_p$. (Note that $\mathcal{H}_{g}^{g} = \mathcal{H}_g$.) In the case of $f = g - 1$, we refer to $\mathcal{H}_{g}^{g-1}$ as the \textbf{non-ordinary locus of $\mathcal{H}_{g}$.}

Our goal is to collect data on the $p$-ranks of hyperelliptic curves over $\mathbb{F}_q$ up to geometric isomorphism.  In turn, this data allows us to draw conclusions primarily about the geometric structure of $\mathcal{H}_{g}^{g-1}$.  Our method also enables us to investigate the spaces $\mathcal{H}_{g}^{f}$ for $f < g-1$.  While \cite{GP05}*{Theorem 2.3} shows that the Deligne-Mumford compactification $\overline{\mathcal{H}}_g^f$ is pure of codimension $g-f$ and that $\mathcal{H}_g^f$ sits densely inside its compactification, we are particularly interested in determining the number of irreducible components of $\mathcal{H}_g^f$.

\subsection{The Hasse-Witt Matrix}\label{Sec:HasseWitt}
The first step towards our goal involves sampling hyperelliptic curves over $\mathbb{F}_q$, which we detail in \cref{Sec:CurveGen}.  Once we have a sample of hyperelliptic curves, we then need to calculate the $p$-rank for each curve.  To do this, we use a corresponding Hasse-Witt matrix.  

Suppose that $C/\mathbb{F}_q$ is a hyperelliptic curve of genus $g$ over $\mathbb{F}_q$, where $q = p^r$. A \textbf{Hasse-Witt matrix} of $C$ is a $g \times g$-matrix corresponding to the action of the Frobenius operator on a given basis of $\mathrm{H}^1(C, \mathcal{O}_{C})$. Let $A$ denote a Hasse-Witt matrix of $C$, and let $\sigma$ denote the action of the $p\textsuperscript{th}$ power map on $A$, i.e. $A^{\sigma} = (a_{i, j}^p)$. It follows from the main result of \cite{Oda69} that the $p$-rank of $C$ is equal to the rank of the $g\textsuperscript{th}$ semilinear power of $A$, which can be written as
\[ A^{(g)} = A A^{\sigma} \dotsm A^{\sigma^{g-1}}. \]
Note that when $C$ is defined over $\mathbb{F}_p$, we have $A^{(g)} = A^g$.  To compute a Hasse-Witt matrix in practice, we use the following proposition.

\begin{proposition}[\cite{Yui78}, clarified by \cite{AH19}]
    Let $C/\mathbb{F}_q$ be a hyperelliptic curve with affine model $y^2 = f(x)$, and define the constants $c_m \in \mathbb{F}_q$ by 
    \[f(x)^{\frac{p-1}{2}} = \sum_m c_m x^m.\] 
    Then, the $g \times g$-matrix $A = (a_{i,j})$ defined by $a_{i,j} = c_{jp - i}$ for all $0 \leq i, j \leq g$ is a Hasse-Witt matrix of $C$.
\end{proposition}

\begin{proof}
    Let $B$ be a $g \times g$-matrix corresponding to the action of the Verschiebung operator on a basis of $\mathrm{H}^0(C, \Omega_{C}^{1})$, commonly called a Cartier-Manin matrix of $C$.  If we choose the basis $\{x^{i-1}\frac{dx}{y} : 1 \leq i \leq g\}$ for $\mathrm{H}^0(C, \Omega_C^1)$, then we can compute a Cartier-Manin matrix $B = (b_{i,j})$ 
    by writing 
    \[f(x)^{\frac{p-1}{2}} = \sum_m c_m x^m\] 
    and taking $b_{i,j} = c_{ip-j}^{1/p}$, as initially described in \cite{Yui78} and later clarified in \cite{AH19}*{Section\ 3.1}.  Since $B$ is a Cartier-Manin matrix of $C$, $A = (B^{\sigma})^T = (c_{jp - i})_{1\leq i,j\leq g}$ is a Hasse-Witt matrix of $C$ (see \cite{AH19}*{Section\ 2.4}).
\end{proof}

\subsection{Lang-Weil Bounds}\label{Sec:LangWeil}

Having computed $p$-ranks for the curves in our sample, we next need to turn this data about $p$-ranks into information about the geometry of $\mathcal{H}_g^f$.  In order to obtain heuristics for the number of components of $\mathcal{H}_g^f$, we make use of the Lang-Weil bounds.

\begin{theorem}[\cite{LW54}*{Cor.\ 2}] \label{thm: Lang-Weil}
Let $\mathbb{F}_q$ be a finite field and let $V \subseteq \mathbb{P}^n$ be an irreducible variety of dimension $d$ defined over $\mathbb{F}_q$. Denote by $V(\mathbb{F}_q)$ the set of $\mathbb{F}_q$-rational points of $V$. Then, we have
    \begin{equation*} \label{eq: L-W irred}
    \# V(\mathbb{F}_q) = q^{d} + O\left(q^{d-\frac{1}{2}}\right).
    \end{equation*}
\end{theorem}
As a consequence, if $V/\mathbb{F}_q$ has $c_{V,\mathbb{F}_q}$ top-dimensional irreducible components, i.e.\ geometrically irreducible components that can be defined over $\mathbb{F}_q$, then we have
\begin{equation*} \label{eq: L-W red}
\# V(\mathbb{F}_q) = c_{V,\mathbb{F}_q} q^{d} + O\left(q^{d-\frac{1}{2}}\right).
\end{equation*}

For the purposes of this paper, one may take $V=\overline{\mathcal{H}}_g^f$, which is defined over $\mathbb{F}_p$. Let $c_{g,f,p,r}$ be the number of irreducible components of $\overline{\mathcal{H}}_g^f \otimes \Spec(\mathbb{F}_{p^r})$ and let $c_{g,f,p}$ be the number of irreducible components of $\overline{\mathcal{H}}_g^f \otimes \Spec(\overline{\mathbb{F}}_p)$. Then, since $\overline{\mathcal{H}}_g^f$ is pure of dimension $g+f-1$, we have
\begin{align*} \label{eq:L-W Hgf}
    \# \overline{\mathcal{H}}_g^f(\mathbb{F}_{p^r}) = c_{g,f,p,r} p^{r(g+f-1)} + O\left(p^{r\left(g+f-1-\frac{1}{2} \right)}\right).
\end{align*}

Note that $\mathcal{H}_g^f$ is dense inside every component of $\overline{\mathcal{H}}_g^f$, so we find the same number of components if we reduce our search to smooth hyperelliptic curves. Recall also that the full space $\overline{\mathcal{H}}_g$ is geometrically irreducible. Given a sufficiently large sample $S$ from $\mathcal{H}_g(\mathbb{F}_{p^r})$, it follows that
\begin{equation*} \label{Lang-Weil last}
\frac{\#\left(S \cap \mathcal{H}_g^f \right)} {\# S} \approx \frac{\overline{\mathcal{H}}_g^f(\mathbb{F}_{p^r})}{\overline{\mathcal{H}}_g(\mathbb{F}_{p^r})} =\frac{c_{g,f,p,r} p^{r(g+f-1)} + O\left(p^{r\left( g+f-1-\frac{1}{2} \right)}\right)} {p^{r(2g-1)} + O\left(p^{r\left( 2g-1-\frac{1}{2} \right)}\right)} \approx \frac{c_{g,f,p,r}}{p^{r(g-f)}}. 
\end{equation*}
This heuristic immediately results in an approximation of $c_{g,f,p,r}$:
\begin{equation} \label{eq: cgfpr approx}
c_{g,f,p,r} \approx \frac{p^{r(g-f)} \# (S \cap \mathcal{H}_g^f)} {\#S}.
\end{equation}
We note that the approximation will become more accurate as $S$ and $\mathcal{H}_g^f(\mathbb{F}_{p^r})$ grow.  Finally, we have $$\max_{r \leq R} c_{g,f,p,r}=c_{g,f,p}$$ for $R$ sufficiently large, since every geometric component can be defined over a finite field. Thus, this sampling method gives rather strong evidence for the number $c_{g,f,p}$, which is of primary interest.

\section{Computational Methods}\label{Sec:CurveGen}
\subsection{Family Method} \label{Sec:Family}

Our first method for generating samples of hyperelliptic curves involves working with a family of curves whose models we have fine control over. The curves in the sample may be chosen up to $\mathbb{F}_q$-isomorphism class, but unfortunately, are not guaranteed to be geometrically non-isomorphic.  We note that this downside is quite negligible, though, as the vast majority of curves have no twists apart from the quadratic twist.  In fact, by \cite{GV08}*{Proposition 2.1}, the curves with additional twists form a locus of codimension $g-1$ in $\mathcal{H}_g$.  Furthermore, by construction, our family will exclude quadratic twists.

Let $g \geq 2$, $p$ an odd prime, $r \geq 1$, and $q = p^r$. We assume that $p$ does not divide $2g+1$. When the characteristic does divide $2g+1$, we rely on the secondary sampling method outlined in \cref{Sec:GaloisType}.  Over $\overline{\mathbb{F}}_p$, every hyperelliptic curve possesses a rational branch point that can be shifted to infinity via a rational map. Therefore, we focus on hyperelliptic curves over $\mathbb{F}_q$ that have a branch point in $\mathbb{F}_q$. 

\begin{proposition}[\cite{Loc94}*{Prop.\ 1.2}]
    Let $C/\mathbb{F}_q$ be a hyperelliptic curve of genus $g$ with a rational branch point.
    Then, $C$ admits a model of the form $y^2 = f(x)$, where $f \in \mathbb{F}_q[x]$ is a monic polynomial of degree $2g+1$.
    Furthermore, if $p$ does not divide $2g+1$, we can assume that \[f(x) = x^{2g+1}+\tilde{f}(x),\]
    where $\tilde{f}\in\mathbb{F}_q[x]$ is of degree at most $2g-1$.
\end{proposition}

\begin{proof}
    Since $C$ has a rational branch point, we can move it to infinity.  Thus, there exists $f \in \mathbb{F}_q[x]$ of degree $2g+1$ such that $C$ is given by the vanishing of $y^2 = f(x)$.  After renormalization, we may assume that the coefficient of $x^{2g+1}$ in $f$ is $1$. Since $p$ does not divide $2g+1$, a well-chosen linear change of variables can remove the $x^{2g}$ term.
\end{proof}

We call a \textbf{pointed curve} $(C, P)$ a pair consisting of a curve $C$ and a marked point $P$ on $C$.
We now consider the set of pointed hyperelliptic curves with their marked point at infinity. 
Any isomorphism between two such curves must preserve the point at infinity, 
and isomorphisms of this form can be explicitly described as in the following proposition.

\begin{proposition}[\cite{Loc94}*{Prop.\ 1.2}]\label{prop:iso_pointed}
    Let $(C_1,\infty)/\mathbb{F}_q$ and $(C_2,\infty)/\mathbb{F}_q$ be two pointed hyperelliptic curves defined by \[y^2 = x^{2g+1}+f_i(x)\,\] where $\deg(f_i) \leq 2g-1$. 
    Then, $(C_1,\infty)$ and $(C_2,\infty)$ are isomorphic if and only if there exists $u\in\overline{\mathbb{F}}_p$ such that \[f_2(x) = u^{2g+1}f_1\left(\frac{x}{u}\right).\]
\end{proposition}

Let us consider the subset of pointed hyperelliptic curves with marked points at infinity, where each curve has an affine model of the form \[y^2 = x^{2g+1}+\alpha_1 x^{2g - 1} + \alpha_2 x^{2g - 2} + \dotsb,\]
such that $\alpha_1 = \alpha_2$ and both coefficients are non-zero.  We denote this set $\mathcal{F}_{g,q}$.

\begin{corollary}\label{cor:unicity_family}
    Let $(C_1,\infty)$ and $(C_2,\infty)$ be elements of $\mathcal{F}_{g,q}$ defined by $f_1$ and $f_2$ with $f_1\neq f_2$. Then, these pointed hyperelliptic curves are not geometrically isomorphic. 
\end{corollary}

This construction yields a quasi-affine subscheme of dimension $2g-1$ of the moduli space of pointed hyperelliptic curves $\mathcal{H}_{g,1}$ (which has dimension $2g$).  We now wish to separate the curves in $\mathcal{F}_{g,q}$ by their $\mathbb{F}_q$-isomorphism classes. 

\begin{proposition} \label{prop:algo}
    Let $(C, \infty) \in \mathcal{F}_{g,q}$. Then, there are at most $2g + 1$ curves $(C', \infty) \in \mathcal{F}_{g,q}$ such that $C$ is $\mathbb{F}_q$-isomorphic to $C'$. Moreover, the collection of all such $\mathbb{F}_q$-isomorphic curves can be explicitly computed.
\end{proposition}

\begin{proof}
    Let $(C,\infty)\in \mathcal{F}_{g,q}$ be given by $y^2 = h(x) = x^{2g + 1} + f(x)$.
    To find a pointed curve $(C',\infty) \in \mathcal{F}_{g,q}$ which is not isomorphic to $(C,\infty)$, but whose curve $C'$ is $\mathbb{F}_q$-isomorphic to $C$, we are led to consider the roots of the polynomial $h$. Indeed, any $\mathbb{F}_q$-isomorphism $\psi$ between $C$ and $C'$ must send the roots of $h$ to the roots of $h'$, where $C'$ is given by $y^2=h'(x)$.

    We note that the root at infinity of $h$ is not sent to the root at infinity of $h'$, as this would preserve the marked point and thus be an isomorphism of pointed curves.  Therefore, another root of $h$ must be sent to infinity.  Let $r$ be such a root of $h$, and let $\varphi_r$ be an $\mathbb{F}_q$-isomorphism which sends $r$ to $\infty$, such that $\varphi_r(h)(x) = x^{2g+1}+\tilde{h}(x)$, with $\deg(\tilde{h})\leq 2g-1$.

    Let $C'' \colon y^2=\varphi_r(h)(x)$.  Then, $\varphi_{r} \circ \psi^{-1} \colon C' \to C''$ is an $\mathbb{F}_q$-isomorphism that preserves the marked point at infinity. Since $C'\in\mathcal{F}_{g,q}$, by \cref{prop:iso_pointed}, the coefficients of the $x^{2g-1}$ term and the $x^{2g-2}$ term in $\tilde{h}(x)$ do not vanish. We can further assume that they are equal, up to a transformation of the form given in \cref{prop:iso_pointed}.  Thus, the pointed curve $(C'',\infty)$ is in $\mathcal{F}_{g,q}$ and is $\mathbb{F}_q$-isomorphic to $(C',\infty)$. We use \cref{cor:unicity_family} to conclude that $\varphi_r(h) = h'$. 

    We observe that the number of curves which are $\mathbb{F}_q$-isomorphic to $C$ is bounded by $2g+1$. Moreover, all such curves can be explicitly computed by finding a transformation that sends a root of $h$ to infinity, and then normalizing the resulting polynomial in the correct way. If at least one of the coefficients for the $x^{2g-1}$ term or the $x^{2g-2}$ term vanishes, then the corresponding $\mathbb{F}_q$-isomorphic curve does not belong to $\mathcal{F}_{g,q}$ and we do not have to consider it.
\end{proof}

By \cref{prop:algo}, given an arbitrary pointed curve $(C, \infty) \in \mathcal{F}_{g,q}$, one can explicitly determine the $\mathbb{F}_q$-isomorphism class of $C$.  This gives rise to the following algorithm, which, given a pointed curve $(C, \infty) \in \mathcal{F}_{g,q}$, returns the smallest pointed curve $(C', \infty) \in \mathcal{F}_{g,q}$ such that $C$ and $C'$ are $\mathbb{F}_q$-isomorphic, up to a fixed, chosen ordering on $\mathbb{F}_{q}$.  Furthermore, one can extend the order on $\mathbb{F}_q$ to a lexicographic order $\mathbb{F}_q[x]$, using $1 < x < x^2 < \cdots$.

\begin{algorithm}\label{alg:Fq_isomclass}
    \mbox{}
    \begin{enumerate}[leftmargin=6em]
        \item[\textsc{Input:}] A polynomial $f(x)$, such that $(C:y^2=f(x),\infty) \in \mathcal{F}_{g,q}$.
        \item[\textsc{Output:}] The smallest polynomial $h(x)$ such that $(C':y^2=h(x),\infty) \in \mathcal{F}_{g,q}$ and $C \cong C'$ over $\mathbb{F}_q$.
    \end{enumerate}
    
    \begin{enumerate}
        \item CurvesIso $\gets [f]$.
        \item Compute the set $R$ of roots of $f$ in $\mathbb{F}_q$.
        \item For each root in $R$ do:
            \begin{enumerate}
                \item Choose a Möbius transformation $M$ which sends that root to $\infty$.
                \item $\tilde{f}\gets f(M^{-1})$.
                \item Renormalize $\tilde{f}$ by its coefficient in $x^{2g+1}$.
                \item $\alpha\gets$ Coefficient of $\tilde{f}$ in $x^{2g}$ divided by $2g+1$. 
                \item $\tilde{f}\gets \tilde{f}(x-\alpha)$.
                \item $a,b =$ Coefficients in $x^{2g-1},x^{2g-2}$ of $\tilde{f}$.
                \item If $ab \neq 0$ then: 
                \begin{enumerate}
                    \item $u\gets a/b$.
                    \item $\tilde{f}\gets u^{2g+1}f(x/u)$.
                    \item Append CurvesIso with $\tilde{f}$.
                \end{enumerate}
            \end{enumerate}
        \item Return the minimum of CurvesIso.
    \end{enumerate}
\end{algorithm}

With the necessary background in hand, we now summarize the \textbf{family method} for sampling hyperelliptic curves. First, we give an explicit ($2g-1$)-dimensional family $\mathcal{F}_{g,q}$ of pointed hyperelliptic curves $(C, \infty)$, which are unique up to geometric isomorphism (see \cref{cor:unicity_family}). We then take a random sample of pointed curves in $\mathcal{F}_{g,q}$, which is easily done since the curves in $\mathcal{F}_{g,q}$ are defined explicitly by their models. Finally, we use \cref{alg:Fq_isomclass} to remove the $\mathbb{F}_q$-isomorphic curves from the sample. See \cref{Sec:Family_Data} for the data generated from this method.

\subsection{Galois Type Method} \label{Sec:GaloisType}

We now turn our attention to an alternate method of sampling hyperelliptic curves to use in the cases when the family method is not applicable and as reinforcement of the family method results.  \cite{How25} describes an algorithm to enumerate all hyperelliptic curves over $\mathbb{F}_q$ up to $\mathbb{F}_q$-isomorphism for a specified genus. We have implemented a modified version of part of that algorithm to generate data sets of hyperelliptic curves over $\mathbb{F}_q$ up to geometric isomorphism.

\cite{How25}*{Algorithm 7.1} can be used to generate a set of hyperelliptic curves over $\mathbb{F}_q$ with models of the form $y^2 = f(x)$, where
\begin{enumerate}
    \item the curves are unique up to $\mathbb{F}_q$-isomorphism, and
    \item the polynomial $f$ splits completely over $\mathbb{F}_q$.
\end{enumerate}
The second condition inspired the name for this sampling method, as the algorithm enumerates curves with Galois type $(1, \ldots, 1)$. We restrict ourselves to enumerating curves with this Galois type for two reasons. The first is that each curve in the sample will be unique up to geometric isomorphism (as opposed to just unique up to $\mathbb{F}_q$-isomorphism).  The second is that the Lang-Weil bounds can be applied to samples of curves with this Galois type.  To prove the former assertion, we present the following proposition.

\begin{proposition}
    Let $C$ and $C'$ be two hyperelliptic curves of positive genus defined over $\mathbb{F}_q$ such that all of their branch points are in $\mathbb{F}_q$.  Moreover, assume that $C$ and $C'$ are geometrically isomorphic.  Then, $C$ and $C'$ are $\mathbb{F}_q$-isomorphic.
\end{proposition}

\begin{proof}
    Let $\psi \colon C \to C'$ be a (geometric) isomorphism. By \cref{prop:iso_hyp_cur}, the isomorphism $\psi$ is given by a Möbius transformation with corresponding matrix $M$ and an element $e \in \overline{\mathbb{F}}_p^\times$.  Therefore, since $\psi$ must send a branch point of $C$ to a branch point of $C'$, $M$ must send the $2g+2$ branch points of $C$ to those of $C'$. It is known that fixing the image of three distinct elements in the domain of a Möbius transformation gives a unique corresponding transformation, with an explicit equation. Here, the image of the $2g+2 \geq 3$ branch points must lie in $\mathbb{F}_q$, thus giving a rational Möbius transformation. It follows that $e$ must be also rational, making $\psi$ an $\mathbb{F}_q$-isomorphism.
\end{proof}

To show the latter assertion, i.e.\ that the Lang-Weil bounds apply to samples of curves with Galois type $(1, \ldots ,1)$, we make the following argument.  Note that there exists an isomorphism 
\[ \varphi \colon (\text{Sym}^{2g+2}(\mathbb{P}^1)\setminus \Delta_W)/\mathrm{PGL}_2 \to \mathcal{H}_g, \]
where $\Delta_W$ is the weak diagonal. The map $\varphi$ sends the class of $(a_1, \ldots ,a_{2g+2})$ to the isomorphism class of the hyperelliptic curve branched at $a_1, \ldots , a_{2g+2}$. If $P \in (\text{Sym}^{2g+2}(\mathbb{P}^1) \setminus \Delta_W)/\mathrm{PGL}_2$ is $\mathbb{F}_{p^r}$-rational, then so is $\varphi(P)$, but the converse does not hold in general. We observe that every hyperelliptic curve has Galois type $(1,\ldots, 1)$ over some field, which may not be the field of definition of the curve. Hence, the space $\mathcal{H}_g^f$ can have more components over $\mathbb{F}_{p^r}$ than the $p$-rank $f$ stratum of $(\text{Sym}^{2g+2}(\mathbb{P}^1) \setminus \Delta_W)/\mathrm{PGL}_2$, but they have the same number of geometric components because of the isomorphism between them. Therefore, the Lang-Weil bounds provide that, for $R$ sufficiently large, samples $S$ with Galois type $(1, \ldots , 1)$ over $\mathbb{F}_{p^r}$, with $r\leq R$, satisfy 
\[ c_{g,f,p} \approx \max_{r\leq R} \frac{p^{r(g-f)} \# (S \cap \mathcal{H}_g^f)}{ \# S}, \]
Thus, we implement a modified version of \cite{How25}*{Algorithm 7.1} to generate data sets of hyperelliptic curves.

\begin{algorithm}{\label{GaloisTypeAlg}}
    Let $p$ be an odd prime, $r \geq 1$ an integer, and $q = p^r$.

    \begin{enumerate}[leftmargin=6em]
        \item[\textsc{Input:}] An integer $g \in \{3, 4, 5\}$, a prime power $q > 2g + 1$, and an integer $s \geq 1$.
        \item[\textsc{Output:}] A list $L$ of size $s$ of polynomials $f(x) \in \mathbb{F}_q[x]$ that split completely over $\mathbb{F}_q$, and which represent hyperelliptic curves $y^2 = f(x)$ of genus $g$ that are unique up to geometric isomorphism.
    \end{enumerate}

    \begin{enumerate}
        \item Let $L$ be the empty list, set $a_1 := \infty$ and $a_2 := 0$, and define $n = 2g + 2$.
        \item For every set $\{ a_3, \dotsc, a_n \}$ of distinct elements of $\mathbb{F}_q \setminus \{ 0 \}$ do:
        \begin{enumerate}
            \item Set $f := z \prod_{i=2}^{n} (x - a_i z)$.
            \item Set $F := \{ \Gamma(f) \}$, where $\Gamma$ ranges over elements of $\PGL_2(\mathbb{F}_q)$ that send two elements of $\{a_i\}$ to $\infty$ and $0$.
            \item If $f$ is the smallest element of $F$ under the ordering $<$, and the automorphism group of $f$ is trivial, append $f$ to $L$.
            \item If $\#L = s$, return $L$. Otherwise continue to the next set $\{a_i\}$.
        \end{enumerate}
        \item Return $L$.
    \end{enumerate}

\end{algorithm}

We note that, unlike the algorithm in \cite{How25}, we exclude hyperelliptic curves with extra automorphisms. We do so for two reasons. First, excluding these curves was less computationally expensive than including them, and only slight modifications to the original code from \cite{How25} were necessary to exclude them. Second, the exclusion of curves with extra automorphisms does not result in a diluted data set. By \cite{GV08}*{Prop. 2.1}, the locus of hyperelliptic curves with extra automorphisms is of codimension $g - 1$ in $\mathcal{H}_g$. The non-ordinary locus $\mathcal{H}_{g}^{g-1}$ is of codimension $1$ in $\mathcal{H}_g$, so for $g \geq 3$ the non-ordinary locus is much larger. Thus, the curves with extra automorphisms do not constitute a significant percentage of the total number of non-ordinary curves in the cases we consider. 

We also note that this method does not produce a random sample of hyperelliptic curves, as we return a sample of size $s$ whose roots are the smallest possible under the ordering on $\mathbb{F}_q$. However, we do not expect that this enumeration method biases our data towards particular $p$-rank strata. See \cref{Sec:GT_Data} for our data generated via this method.

\subsection{Computation of $p$-Ranks}

Having collected a sample of hyperelliptic curves, we next need to compute the $p$-ranks of these curves.  In practice, we have implemented the following algorithm to compute the $p$-ranks of hyperelliptic curves using Hasse-Witt matrices. See \cref{Sec:HasseWitt} for more details.

\begin{algorithm}
    Let $C/\mathbb{F}_q$ be a smooth hyperelliptic curve of genus $g$ defined over the finite field $\mathbb{F}_q$ with $q = p^r$ for $p$ prime. Let $C$ have model $y^2 = f(x)$, where $f(x)$ is a polynomial of degree $2g + 1$ or $2g + 2$.
   
    \begin{enumerate}[leftmargin=6em]
        \item[\textsc{Input:}] The polynomial $f(x) \in \mathbb{F}_q[x]$ and prime power $q = p^r$.
        \item[\textsc{Output:}] The $p$-rank of $C:y^2 = f(x)$.
    \end{enumerate}
    
    \begin{enumerate}
        \item Compute $f(x)^{\frac{p-1}{2}} = \sum_{m} c_m x^m$ and $g = \lfloor(\deg(f) - 1)/2 \rfloor$. 
        \item Define the Hasse-Witt matrix $A = (a_{i,j})$ by setting $a_{i,j} = c_{jp-i}$ for all $0 \leq i, j \leq g$.
        \item If $q = p^1$, return the rank of $A^g$, otherwise continue to step (4).
        \item If $q = p^r$ for $r > 1$, return the rank of $A^{(g)} = A A^{\sigma} \dotsm A^{\sigma^{g-1}}$.
    \end{enumerate}
\end{algorithm}

\section{Non-Ordinary Locus} \label{sec:nonord}

In the subsequent section, we provide an overview of the data collected regarding the number of components of the non-ordinary locus.  \cref{Sec:Family_Data,Sec:GT_Data} provide an analysis of the data collected using the sampling methods described in \cref{Sec:Family,Sec:GaloisType}, respectively.  Then, we present our main conjecture based on these findings in \cref{Sec:Conjectures}.

The analysis for both sampling methods followed the same general structure.  First, a sample of curves over $\mathbb{F}_q$ was generated using one of the two algorithms previously discussed; we will denote the number of curves in this sample by $s$.  Then, the $p$-rank for each curve in the sample was calculated using the method described in \cref{Sec:HasseWitt}.  With the $p$-rank data in hand, we could next count the number of non-ordinary curves, which we will denote $N$. 

Using the above information, we computed \[M_{g,p,r} = \frac{N}{s} \cdot p^r\]
for various $p$ and appropriate values of $r$.  (We give explicit ranges for $p$ and values for $r$ in the subsequent sections.)  Recall that our goal is to determine the number of geometrically irreducible components of the non-ordinary locus, $c_{g,g-1,p} \in \mathbb{Z}_{>0}$.  As outlined in \cref{Sec:LangWeil}, we expect the ratio of non-ordinary curves to sample curves to be proportional to $c_{g, g-1, p, r}$, as defined in \cref{eq: cgfpr approx}.  Therefore, by computing $M_{g,p,r}$, we can approximate $c_{g,g-1,p,r}$, which in turn allows us to find  
\[c_{g,g-1,p}\ = \max_{r\leq R} c_{g,g-1,p,r} \approx \max_{r\leq R} M_{g,p,r}\]
for sufficiently large $R$.

We then calculated several statistics on the data sets consisting of values of $M_{g,p,r}$ for every prime $p$ in the prime range for a given field of definition.  We determined the median and mean of these data sets to estimate the number of geometrically irreducible components.  We also found the minimum, maximum, and standard deviation of the data sets to emphasize the uniformity of the data.

\subsection{Family Method Data Analysis}  \label{Sec:Family_Data}
This section details the data collected using the sampling method described in \cref{Sec:Family}.  \cref{table:Family3} gives the data for the genus $g = 3$ case and \cref{table:Family4} for the $g = 4$ case.  We were also able to use this method to collect data for the cases of $5 \leq g \leq 20$.  The data collected for these genera indicate similar phenomena to those observed in the $g = 3$ and $g = 4$ cases, so we use these small genera as examples.  The complete data set can be found in \cite{Code}.  

\begin{table}[b]
    \small
    \caption{Analysis of the data collected using the family sampling method for $g = 3$.}
    \label{table:Family3}
    \tabcolsep=10pt
    \def\arraystretch{1.2}
    \renewcommand{\familydefault}{\ttdefault}
    \makebox[\textwidth][c]{
    \begin{tabular}{cc S[table-format=9.0]  S[table-format=1.3] S[table-format=1.3] S[table-format=1.3] S[table-format=1.3] S[table-format=1.3]}
        \toprule
        {Field} & {Prime Range} & {Sample Size} & {Median} & {Mean} & {Min.} & {Max.} & {Std. Dev.}\\
        \midrule 
        $\mathbb{F}_{p}$ & 5 - 1000 & 10000000* & 1.012 & 1.016 & 0.982 & 1.092 & 0.020 \\
        $\mathbb{F}_{p^2}$ & 3 - 300 & 10000000 & 1.006 & 1.009 & 0.877 & 1.161 & 0.045 \\
        $\mathbb{F}_{p^3}$ & 3 - 100 & 100000000 & 0.995 & 0.985 & 0.794 & 1.092 & 0.059 \\
        $\mathbb{F}_{p^4}$ & 3 - 30 & 100000000 & 1.002 & 1.005 & 0.976 & 1.054 & 0.026 \\
        \bottomrule
    \end{tabular}
    }
\end{table}

\begin{table}[t]
    \small
    \caption{Analysis of the data collected using the family sampling method for $g = 4$.}
    \label{table:Family4}
    \tabcolsep=10pt
    \def\arraystretch{1.2}
    \renewcommand{\familydefault}{\ttdefault}
    \makebox[\textwidth][c]{
    \begin{tabular}{cc S[table-format=9.0]  S[table-format=1.3] S[table-format=1.3] S[table-format=1.3] S[table-format=1.3] S[table-format=1.3]}
        \toprule
        {Field} & {Prime Range} & {Sample Size} & {Median} & {Mean} & {Min.} & {Max.} & {Std. Dev.}\\
        \midrule 
        $\mathbb{F}_{p}$ & 5 - 1000 & 10000000* & 1.001 & 1.001 & 0.979 & 1.027 & 0.007 \\
        $\mathbb{F}_{p^2}$ & 5 - 300 & 10000000 & 1.001 & 0.998 & 0.831 & 1.118 & 0.049 \\
        $\mathbb{F}_{p^3}$ & 5 - 100 & 100000000 & 1.001 & 0.999 & 0.912 & 1.058 & 0.037 \\
        $\mathbb{F}_{p^4}$ & 5 - 30 & 100000000 & 1.001 & 0.998 & 0.954 & 1.049 & 0.032 \\
        \bottomrule
    \end{tabular}
    }
\end{table}

Further, we note that for $r = 1$ in each genus, we were able to collect data for the entire family of curves for a small number of primes $p$.  The sample sizes in these cases are noted with an asterisk (*), as the number of curves in the entire family exceeded the sample size listed for certain $p$.  Otherwise, $s$ is the upper bound of the sample size, used for the majority of the primes $p$.  For small $p$, the number of curves in the complete family does not always reach the intended sample size.  For example, for $g = 4$ and $r = 1$, we were able to compute the complete family for $p \leq 23$.  For $p = 5$, the entire family has only \texttt{38\,556} curves.  However, for $p = 23$, the entire family has \texttt{2\,065\,333\,130} curves, which far exceeds $s$. 

To further support our data analysis, \cref{fig:1} illustrates the values of $M_{3, p, r}$ that we calculated for each prime $p$ and the appropriate value of $r$.  We note that for the majority of data points, the percent error, i.e.\ the difference between the calculated $M$-value and the expected value of $c_{g,g-1,p} = 1$, is less than 5\%. 

\begin{figure}[b]
\begin{subfigure}[b]{0.49\textwidth}
\begin{tikzpicture}
\begin{axis}[
    title={$g = 3$, $r = 1$},
    width={0.95\textwidth},
    xlabel={$p$},
    xmin=0, xmax=1000,
    ymin=0.75, ymax=1.25,
    ymajorgrids=true,
    grid style=dashed,
    ytick = {0.8, 0.9, 1, 1.1, 1.2}, 
    point meta min=0.75,
    point meta max=1.25,
    point meta=y,
    ylabel={$M_{3, p, r}$},
    ylabel style={rotate=-90}
]
\addplot+[
    colormap name=blueorange-mid,
    colormap access=piecewise constant, 
    scatter,
    mark=*,
    only marks,
    ]
    coordinates {
        (5, 0.981595092)
        (11, 1.030507527)
        (13, 1.026386472)
        (17, 1.027388144)
        (19, 1.037666921)
        (23, 1.049731092)
        (27, 0.9787638195)
        (29, 1.035974537)
        (31, 1.044213984)
        (37, 1.003624378)
        (41, 1.035050649)
        (43, 1.014052024)
        (47, 1.050014815)
        (53, 1.0212729)
        (59, 1.0671743)
        (61, 1.025105)
        (67, 1.0062998)
        (71, 1.0924841)
        (73, 1.0056261)
        (79, 1.0551872)
        (83, 1.0507219)
        (89, 1.0237403)
        (97, 1.0022816)
        (101, 1.0374215)
        (103, 1.0386314)
        (107, 1.0266864)
        (109, 1.0112911)
        (113, 1.0118359)
        (127, 1.0359009)
        (131, 1.0603926)
        (137, 1.0171017)
        (139, 1.0248331)
        (149, 1.0185342)
        (151, 1.0416433)
        (157, 0.9987869)
        (163, 0.9877637)
        (167, 1.060617)
        (173, 1.0076904)
        (179, 1.0310758)
        (181, 1.0078442)
        (191, 1.0699438)
        (193, 0.9980802)
        (197, 1.0042075)
        (199, 1.0381432)
        (211, 1.0180328)
        (223, 1.0303492)
        (227, 1.0195478)
        (229, 0.9925776)
        (233, 1.0138762)
        (239, 1.0538227)
        (241, 1.0061027)
        (251, 1.0427544)
        (257, 1.0130426)
        (263, 1.0526575)
        (269, 1.0123815)
        (271, 1.0325642)
        (277, 0.9998869)
        (281, 1.0146348)
        (283, 1.0087535)
        (293, 1.0213687)
        (307, 1.0088634)
        (311, 1.0656726)
        (313, 0.9999724)
        (317, 0.9989938)
        (331, 1.0139854)
        (337, 0.9992724)
        (347, 1.0171958)
        (349, 1.0153457)
        (353, 1.0026965)
        (359, 1.0454439)
        (367, 1.0227923)
        (373, 1.0039668)
        (379, 1.0001052)
        (383, 1.0356703)
        (389, 1.0149788)
        (397, 0.9847188)
        (401, 1.015733)
        (409, 1.012275)
        (419, 1.0262986)
        (421, 0.999875)
        (431, 1.0541398)
        (433, 0.9920463)
        (439, 1.0400788)
        (443, 0.9872255)
        (449, 1.0055804)
        (457, 0.9934266)
        (461, 1.0181646)
        (463, 1.0018857)
        (467, 1.0066652)
        (479, 1.0532731)
        (487, 1.0059959)
        (491, 1.0259445)
        (499, 1.0043373)
        (503, 1.037689)
        (509, 0.9996251)
        (521, 0.999278)
        (523, 1.0156137)
        (541, 1.0033386)
        (547, 1.0065894)
        (557, 0.9924626)
        (563, 1.0321479)
        (569, 1.009975)
        (571, 0.9987932)
        (577, 1.0083075)
        (587, 1.0106379)
        (593, 1.0193077)
        (599, 1.0351918)
        (601, 1.0083578)
        (607, 1.01369)
        (613, 0.9994352)
        (617, 1.0054015)
        (619, 1.007113)
        (631, 1.0295396)
        (641, 1.0093186)
        (643, 0.9864263)
        (647, 1.0256244)
        (653, 0.9919723)
        (659, 1.017496)
        (661, 1.0055132)
        (673, 0.99604)
        (677, 1.0022308)
        (683, 0.9868667)
        (691, 1.008169)
        (701, 1.0041124)
        (709, 0.986928)
        (719, 1.0400335)
        (727, 1.0193994)
        (733, 1.0019377)
        (739, 1.0227021)
        (743, 1.0130805)
        (751, 1.0118223)
        (757, 0.9876579)
        (761, 1.0036068)
        (769, 1.0097739)
        (773, 0.9899811)
        (787, 1.0027167)
        (797, 1.005017)
        (809, 1.0294525)
        (811, 0.9969623)
        (821, 1.0085164)
        (823, 0.9847195)
        (827, 1.0013316)
        (829, 1.0214109)
        (839, 1.0432126)
        (853, 0.9970717)
        (857, 1.0092889)
        (859, 0.9960105)
        (863, 1.0190304)
        (877, 0.9942549)
        (881, 1.0065425)
        (883, 0.993375)
        (887, 1.03779)
        (907, 0.9952511)
        (911, 1.016676)
        (919, 1.0135651)
        (929, 1.0097301)
        (937, 1.0150521)
        (941, 1.0144921)
        (947, 1.0008843)
        (953, 1.0213301)
        (967, 1.0151566)
        (971, 1.0241137)
        (977, 0.9891148)
        (983, 1.007575)
        (991, 1.0007118)
        (997, 0.9881267)
    };
\end{axis}
\end{tikzpicture}
\end{subfigure}
\hfill
\begin{subfigure}[b]{0.49\textwidth}
\begin{tikzpicture}
\begin{axis}[
    title={$g = 3$, $r = 2$},
    width={0.95\textwidth},
    xlabel={$p$},
    xmin=0, xmax=300,
    ymin=0.75, ymax=1.25,
    ymajorgrids=true,
    grid style=dashed,
    ytick = {0.8, 0.9, 1, 1.1, 1.2}, 
    point meta min=0.75,
    point meta max=1.25,
    point meta=y,
]
\addplot+[
    colormap name=blueorange-mid,
    colormap access=piecewise constant, 
    scatter,
    mark=*,
    only marks,
    ]
    coordinates {
        (3, 0.9337354821)
        (5, 1.000328754)
        (11, 1.002727)
        (13, 1.0011222)
        (17, 1.0033213)
        (19, 1.005385)
        (23, 1.0051529)
        (29, 1.004154)
        (31, 1.0022269)
        (37, 0.9748649)
        (41, 0.9956563)
        (43, 1.0117728)
        (47, 1.0263014)
        (53, 1.0134872)
        (59, 0.9924331)
        (61, 0.9842045)
        (67, 0.983091)
        (71, 0.9789622)
        (73, 1.0306286)
        (79, 0.9873262)
        (83, 0.9520598)
        (89, 1.0146801)
        (97, 1.0302855)
        (101, 1.0129593)
        (103, 1.0343775)
        (107, 0.9891936)
        (109, 0.9873111)
        (113, 1.0662115)
        (127, 0.9919335)
        (131, 1.0279439)
        (137, 1.051064)
        (139, 1.0375377)
        (149, 0.9724038)
        (151, 0.9188803)
        (157, 1.0180037)
        (163, 0.9777392)
        (167, 1.0235263)
        (173, 1.0834298)
        (179, 1.0124956)
        (181, 1.0286954)
        (191, 1.0871338)
        (193, 1.005723)
        (197, 1.0051531)
        (199, 0.9781447)
        (211, 1.0729561)
        (223, 1.094038)
        (227, 1.0460387)
        (229, 1.0593082)
        (233, 1.0152043)
        (239, 0.9482086)
        (241, 0.8770231)
        (251, 1.0017159)
        (257, 1.0633889)
        (263, 1.0444519)
        (269, 0.9045125)
        (271, 1.0061417)
        (277, 1.0128228)
        (281, 1.1607267)
        (283, 1.0171303)
        (293, 0.9872635)
    };
\end{axis}
\end{tikzpicture}
\end{subfigure}
\newline

\begin{subfigure}[b]{0.49\textwidth}
\begin{tikzpicture}
\begin{axis}[
    title={$g = 3$, $r = 3$},
    width={0.95\textwidth},
    xlabel={$p$},
    xmin=0, xmax=100,
    ymin=0.75, ymax=1.25,
    ymajorgrids=true,
    grid style=dashed,
    ytick = {0.8, 0.9, 1, 1.1, 1.2}, 
    point meta min=0.75,
    point meta max=1.25,
    point meta=y,
    ylabel={$M_{3, p, r}$},
    ylabel style={rotate=-90}
]
\addplot+[
    colormap name=blueorange-mid,
    colormap access=piecewise constant, 
    scatter,
    mark=*,
    only marks,
    ]
    coordinates {
        (3, 0.9787638195)
        (5, 1.00020125)
        (11, 1.00079221)
        (13, 0.99488948)
        (17, 1.01217626)
        (19, 1.0144461)
        (23, 1.02214967)
        (29, 0.99092507)
        (31, 0.99323194)
        (37, 1.04193221)
        (41, 0.99453003)
        (43, 0.97396075)
        (47, 0.98424204)
        (53, 0.99300959)
        (59, 1.01046468)
        (61, 1.01460507)
        (67, 1.01357131)
        (71, 1.09162855)
        (73, 0.8947391)
        (79, 0.79379279)
        (83, 0.93201281)
        (89, 0.90236032)
        (97, 1.0039403)
    };
\end{axis}
\end{tikzpicture}
\end{subfigure}
\hfill
\begin{subfigure}[b]{0.49\textwidth}
\begin{tikzpicture}
\begin{axis}[
    title={$g = 3$, $r = 4$},
    width={0.95\textwidth},
    xlabel={$p$},
    xmin=0, xmax=30,
    ymin=0.75, ymax=1.25,
    ymajorgrids=true,
    grid style=dashed,
    ytick = {0.8, 0.9, 1, 1.1, 1.2}, 
    point meta min=0.75,
    point meta max=1.25,
    point meta=y,
]
\addplot+[
    colormap name=blueorange-mid,
    colormap access=piecewise constant, 
    scatter,
    mark=*,
    only marks,
    ]
    coordinates {
        (3, 0.99257643)
        (5, 0.9999625)
        (11, 1.00393337)
        (13, 0.97935669)
        (17, 0.97552528)
        (19, 1.0034717)
        (23, 1.02981488)
        (29, 1.05384869)
    };
\end{axis}
\end{tikzpicture}
\end{subfigure}
\caption{Values of $M_{3,p,r}$ for each prime $p$.}
\label{fig:1}
\end{figure}
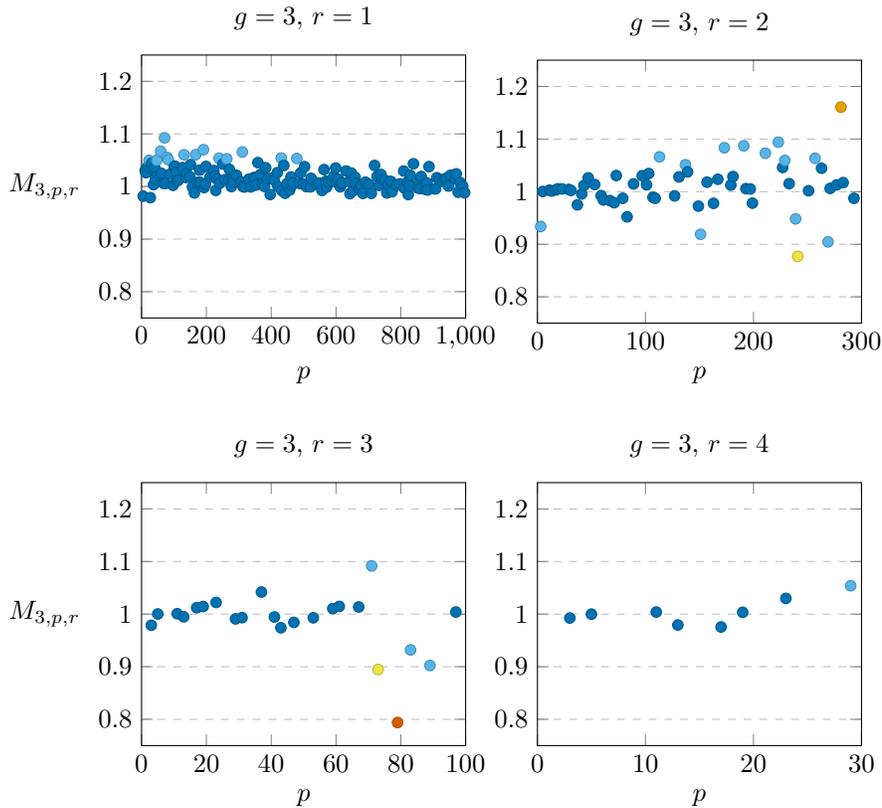

\subsection{Galois Type Method Data Analysis} \label{Sec:GT_Data}
\cref{table:GT3,,table:GT4} communicate the data collected using the Galois type method discussed in \cref{Sec:GaloisType} for the cases $g = 3$ and $g = 4$, respectively.  We also used this method in the $g = 5$ case and found similar statistics; the data can be found in \cite{Code}.  The Galois type method is more computationally expensive than the family method, so while this method could be extended to higher genera, we primarily used it to reinforce the results from the family method and expect similar behavior in higher genera.  

\begin{table}[b]
    \small
    \caption{Analysis of the data collected using the Galois type sampling method for $g = 3$.}
    \label{table:GT3}
    \tabcolsep=10pt
    \def\arraystretch{1.2}
    \renewcommand{\familydefault}{\ttdefault}
    \makebox[\textwidth][c]{
    \begin{tabular}{cc S[table-format=9.0]  S[table-format=1.3] S[table-format=1.3] S[table-format=1.3]  S[table-format=1.3]  S[table-format=1.3]}
        \toprule
        {Field} & {Prime Range} & {Sample Size} & {Median} & {Mean} & {Min.} & {Max.} & {Std. Dev.}\\
        \midrule 
        $\mathbb{F}_{p^2}$ & 17 - 300 & 10000000 & 1.006 & 1.006 & 0.858 & 1.174 & 0.056 \\
        $\mathbb{F}_{p^3}$ & 7 - 100 & 100000000 & 1.000 & 0.999 & 0.945 & 1.059 & 0.028 \\
        $\mathbb{F}_{p^4}$ & 5 - 30 & 100000000 & 0.998 & 1.002 & 0.982 & 1.033 & 0.019 \\
        \bottomrule
    \end{tabular}
    }
\end{table}

\begin{table}[b]
    \small
    \caption{Analysis of the data collected using the Galois type sampling method for $g = 4$.}
    \label{table:GT4}
    \tabcolsep=10pt
    \def\arraystretch{1.2}
    \renewcommand{\familydefault}{\ttdefault}
    \makebox[\textwidth][c]{
    \begin{tabular}{cc S[table-format=9.0]  S[table-format=1.3] S[table-format=1.3] S[table-format=1.3] S[table-format=1.3] S[table-format=1.3]}
        \toprule
        {Field} & {Prime Range} & {Sample Size} & {Median} & {Mean} & {Min.} & {Max.} & {Std. Dev.}\\
        \midrule 
        $\mathbb{F}_{p^2}$ & 11 - 300 & 10000000 & 1.008 & 1.009 & 0.896 & 1.228 & 0.053 \\
        $\mathbb{F}_{p^3}$ & 7 - 100 & 100000000 &  1.000 & 0.997 & 0.941 & 1.060 & 0.026 \\
        $\mathbb{F}_{p^4}$ & 5 - 30 & 100000000 & 1.003 & 0.998 & 0.960 & 1.022 & 0.019 \\
        \bottomrule
    \end{tabular}
    }
\end{table}

We note that for this data, the lower bound of the prime range is the first prime at which the full sample size is reached. Therefore, every sample has exactly the size listed.  As in the previous case, we report the median, mean, minimum, maximum, and standard deviation of the data sets consisting of values of $M_{g,p,r}$ for every prime $p$ in the prime range for a given field of definition.

\begin{remark} \label{rmk: GT r=1 abundance}
When the Galois type method is applied over $\mathbb{F}_p$, the results exhibit abnormal behavior.  In particular, there appears to be an abundance of non-ordinary curves whose branch points are defined over $\mathbb{F}_p$ when $p \equiv 3 \bmod 4$.  In this case, the values of $M_{g,p,1}$ sit between $1.5$ and $4$, but are not particularly close to any integer.  Furthermore, we note that this abundance of non-ordinary curves disappears when the Galois type method is applied over $\mathbb{F}_{p^r}$ for $r>1$.  Additionally, when we instead consider all the Galois types over $\mathbb{F}_p$ for $3 \leq p \leq 47$, there is again no difference between $p \equiv 1 \bmod 4$ and $p \equiv 3 \bmod 4$; in both cases, the data suggest irreducibility of the non-ordinary locus.

We have no explanation presently as to why this phenomenon occurs.  Moreover, our queries failed to produce any hyperelliptic curves of $p$-rank $0$ whose branch points were all defined over $\mathbb{F}_p$ when $p \equiv 1 \bmod 4$.  Due to these oddities, we did not use the Galois type method over $\mathbb{F}_p$.  The latter observation, though, prompts the following conjecture.
\end{remark}

\begin{conjecture} \label{conj: 1mod4}
Assume $p \equiv 1 \bmod 4$. Let $C$ be a hyperelliptic curve of genus at least $3$ whose branch points are all defined over $\mathbb{F}_p$. Then, the $p$-rank of $C$ is at least $1$.
\end{conjecture}

\subsection{Main Conjecture} \label{Sec:Conjectures}

The data presented in \cref{Sec:Family_Data} and \cite{Code} provides strong evidence that $\mathcal{H}_g^{g-1} \otimes \Spec(\mathbb{F}_{p^r})$ is irreducible in the following cases:
\begin{itemize}
    \item $3 \leq g \leq 5$, $r=1$ and $5 \leq p < 1000$;
    \item $3 \leq g \leq 5$, $r=2$ and $3 \leq p < 300$;
    \item $3 \leq g \leq 5$, $r=3$ and $3 \leq p < 100$;
    \item $3 \leq g \leq 5$, $r=4$ and $3 \leq p < 30$;
    \item $6 \leq g \leq 20$, $r=1$ and $5 \leq p < 500$; 
    \item $6 \leq g \leq 20$, $r=2$ and $3\leq p < 100$.
\end{itemize}
Additionally, the data presented in \cref{Sec:GT_Data} and \cite{Code} provides strong evidence that the non-ordinary locus of $(\text{Sym}^{2g+2}(\mathbb{P}^1) \setminus \Delta_W)/\mathrm{PGL}_2 \otimes \Spec(\mathbb{F}_{p^r})$ is irreducible in the following cases:
\begin{itemize}
    \item $3 \leq g \leq 5$, $r=2$ and $17 \leq p < 300$;
    \item $3 \leq g \leq 5$, $r=3$ and $7 \leq p < 100$;
    \item $3 \leq g \leq 5$, $r=4$ and $5 \leq p < 30$.
\end{itemize}
Since $(\text{Sym}^{2g+2}(\mathbb{P}^1) \setminus \Delta_W)/\mathrm{PGL}_2$ is geometrically isomorphic to $\mathcal{H}_g$, this data also suggests irreducibility of $\mathcal{H}_g^{g-1} \otimes \Spec(\mathbb{F}_{p^r})$.

Note that when $p^r$ is small, there do not exist enough curves over $\mathbb{F}_{p^r}$ to give a reliable estimate. On the other hand, when $p^r$ is large compared to the sample size, the non-ordinary curves are too rare to give a reliable estimate.  If the number of components of $\mathcal{H}_g^{g-1} \otimes \Spec(\mathbb{F}_{p^r})$ would grow with $g$, $p$ or $r$, one would expect that such a phenomenon should be visible in the samples presented in this section. 

We suspect that the behavior does not change when $g>20$, when $r>4$ or when $p>1000$. For comparison, the irreducibility of $\mathcal{H}_1^0$ becomes apparent when $p^r=5^2$ (see \cite{Sil09}*{Theorem 4.1(c)}) and the irreducibility of $\mathcal{H}_2^0$ becomes apparent already when $p^r=7^2$ (see \cite{IKO86}*{Remark 2.17}).  In summary, given the behavior for smaller genera, we are confident in making the following conjecture.

\begin{conjecture} \label{conj: non-ordinary}
For $g \geq 3$ and $p\geq 3$, the non-ordinary locus $\mathcal{H}_g^{g-1}$ is geometrically irreducible.
\end{conjecture}

Recall that $\mathcal{A}_2^1$ is geometrically irreducible by \cite{Oor01}*{Corollary 1.5}, and the geometric irreducibility of $\mathcal{H}_2^1$ follows from the fact that the generic point of $\mathcal{A}_2^1$ corresponds to the Jacobian of a smooth hyperelliptic curve.  Thus, this conjecture is a generalization of the known result for $g=2$.

\section{Other $p$-Rank Strata} \label{sec:f<g-1}

\subsection{$p$-Rank At Least 1} \label{sec:f=g-2}
A major benefit of our computations is that we not only count curves of $p$-rank $g-1$, but of all $p$-ranks $f \leq g - 1$.  Consequently, we were able to extend our study to other $p$-rank strata.  The method described in \cref{Sec:Family} was also applied to gather evidence for the number of components of $\mathcal{H}_g^{g-2}$ for $3 \leq g \leq 20$. This yielded heuristics that are very similar to the heuristics in \cref{Sec:Family_Data}.  Additionally, the Galois type method described in \cref{Sec:GaloisType} was applied for $g=3$. Again, we were able to exclude curves with extra automorphisms, this time by the following argument. 

Let $\mathcal{H}_g^{\mathrm{aut}}$ be the locus of $\mathcal{H}_g$ consisting of curves whose automorphism group is not $\mathbb{Z}/2\mathbb{Z}$. By \cite{GV08}*{Proposition 2.1}, this locus is irreducible of dimension $g$, and its generic point corresponds to a hyperelliptic curve with an extra involution. Any such curve admits a model $C \colon y^2= h(x^2)$. Then, \cite{KR89}*{Theorem C} yields the following decomposition up to isogeny:
\[\mathrm{Jac}(C) \sim \mathrm{Jac}(D_1) \times \mathrm{Jac}(D_2).\]
Here, $D_1$ and $D_2$ are given by $D_1 \colon y^2 = h(x)$ and $D_2 \colon  y^2 = xh(x)$. By picking $h(x)$ appropriately, we can arrange for $D_2$ to be ordinary; hence, $f_C \geq f_{D_2} = g_{D_2} = \left \lceil \frac{g}{2} \right \rceil$. Thus, the generic point of $\mathcal{H}_g^{\mathrm{aut}}$ corresponds to a curve with $p$-rank at least $2$, implying that $\mathcal{H}_g^{\mathrm{aut}}$ cannot contain a component of $\mathcal{H}_g^f$ when $f\leq 1$. Moreover, $\mathcal{H}_g^{\mathrm{aut}}$ cannot contain a component of $\mathcal{H}_g^f$ when $f\geq 2$ since $\dim(\mathcal{H}_g^f)=g-1+f$.

To count the number of curves with $p$-rank $f \leq g-2$, we employed a similar method to that described at the start of \cref{sec:nonord}.  Since $\mathcal{H}_g^{g-2}$ has codimension $2$, for each sample $S \subseteq \mathcal{H}_g(\mathbb{F}_{p^r})$, we computed the ratio 
\begin{equation*} \label{eq:Mg-2}
M_{g,g-2,p,r} = \frac{\# (\mathcal{S} \cap \mathcal{H}_g^{g-2})} {\#S} p^{2r}.
\end{equation*}
We then calculated the median, mean, minimum, maximum, and standard deviation on the data sets consisting of values of $M_{g,g-2,p,r}$ for every prime $p$ in a certain prime range for a given field of definition.  

\begin{table}[b]
    \small
    \caption{Statistics for data sets of $M_{g, g-2, p, 1}$ for $3 \leq g \leq 20$ collected using the family method.}
    \label{table:f=g-2}
    \tabcolsep=10pt
    \def\arraystretch{1.2}
    \renewcommand{\familydefault}{\ttdefault}
    \makebox[\textwidth][c]{
    \begin{tabular}{cc S[table-format=8.0]  S[table-format=1.3] S[table-format=1.3] S[table-format=1.3] S[table-format=1.3] S[table-format=1.3]}
        \toprule
        {Genus} & {Prime Range} & {Sample Size} & {Median} & {Mean} & {Min.} & {Max.} & {Std. Dev.}\\
        \midrule 
        3 & 5 - 1000 & 10000000* & 0.992 & 0.974 & 0.294 & 1.643 & 0.195 \\
        4 & 5 - 1000 & 10000000* & 1.014 & 0.999 & 0.381 & 1.604 & 0.182 \\
        5 & 3 - 1000 & 10000000* & 1.010 & 0.992 & 0.432 & 1.704 & 0.175 \\
        6 & 5 - 500 & 10000000 & 1.036 & 1.031 & 0.219 & 2.062 & 0.287 \\
        7 & 7 - 500 & 10000000 & 1.030 & 1.012 & 0.202 & 1.688 & 0.244 \\
        8 & 3 - 500 & 10000000 & 1.009 & 1.035 & 0.425 & 1.930 & 0.272 \\
        9 & 3 - 500 & 10000000 & 1.046 & 1.066 & 0.303 & 2.313 & 0.304 \\
        10 & 5 - 500 & 10000000 & 1.007 & 0.962 & 0.429 & 1.867 & 0.270 \\
        11 & 3 - 500 & 10000000 & 1.037 & 0.990 & 0.322 & 1.913 & 0.266 \\
        12 & 3 - 500 & 10000000 & 1.005 & 1.020 & 0.294 & 2.062 & 0.256 \\
        13 & 5 - 500 & 10000000 & 0.987 & 1.030 & 0.213 & 2.118 & 0.313 \\
        14 & 3 - 500 & 10000000 & 1.044 & 1.018 & 0.000 & 1.756 & 0.276 \\
        15 & 3 - 500 & 10000000 & 1.043 & 1.051 & 0.269 & 1.745 & 0.251 \\
        16 & 5 - 500 & 10000000 & 1.034 & 1.004 & 0.144 & 1.766 & 0.308 \\
        17 & 3 - 500 & 10000000 & 0.998 & 0.965 & 0.229 & 1.809 & 0.266 \\
        18 & 3 - 500 & 10000000 & 0.991 & 0.999 & 0.201 & 1.816 & 0.235 \\
        19 & 5 - 500 & 10000000 & 0.997 & 0.996 & 0.229 & 1.913 & 0.250 \\
        20 & 3 - 500 & 10000000 & 1.042 & 1.020 & 0.218 & 2.372 & 0.274 \\
        \bottomrule
    \end{tabular}
    }
\end{table}

\cref{table:f=g-2} illustrates the above statistics for $M_{g,g-2,p,1}$ for $3 \leq g \leq 20$ collected using the family method.  As mentioned above, we also computed similar statistics using the Galois type method.  These statistics corroborated the results determined using the family method and so are not included in the table; they can be found in full in \cite{Code}.  Similar to \cref{table:Family3,table:Family4}, we note that since $r = 1$, we were able to collect data for the entire family of curves for a small number of primes $p$.  The sample sizes in these cases are once again noted with an asterisk (*) to indicate that for certain $p$, the entire family exceeded the sample size listed. 

\cref{table:f=g-2} suggests irreducibility of $\mathcal{H}_g^{g-2} \otimes \Spec(\mathbb{F}_{p})$ for the considered values of $g$ and $p$. Additional data for $r = 2$ is available in \cite{Code} and exhibits similar behavior to the $r = 1$ case for the smaller prime ranges we could compute.  We note that the minima are smaller and the maxima are larger in this data than in the non-ordinary data.  We attribute this to the rarity of curves with desired $p$-rank.

\cref{fig:f=g-2} demonstrates how the error of $M_{3,1,p,1}$ behaves. The spread of the values of $M_{3,1,p,1}$ widens as $p$ grows, simply because the curves of $p$-rank at most $g-2$ become increasingly rare.  For example, in the $g = 3$ case, there are \texttt{1001} curves with $p$-rank $f \leq 1$ when $p = 103$, but only \texttt{10} such curves when $p = 883$.  Given this explanation along with the data available, we feel confident making the following conjecture.

\begin{figure}[b]
\begin{tikzpicture}
\begin{axis}[
    width={0.9\textwidth},
    xlabel={$p$},
    xmin=0, xmax=1000,
    ymin=0, ymax=2,
    ymajorgrids=true,
    grid style=dashed,
    ytick = {0.25, 0.5, 0.75, 1, 1.25, 1.5, 1.75}, 
    point meta min=0.5,
    point meta max=1.5,
    point meta=y,
    ylabel={$M_{3, 1, p, 1}$},
    ylabel style={rotate=-90}
]
\addplot+[
    colormap name=blueorange-mid,
    colormap access=piecewise constant, 
    scatter,
    mark=*,
    only marks,
    ]
    coordinates {
        (5, 1.107702795)
        (11, 1.092210318)
        (13, 1.076225545)
        (17, 1.030371594)
        (19, 1.033624444)
        (23, 1.015032934)
        (27, 0.9774057341)
        (29, 1.046446924)
        (31, 1.022832069)
        (37, 1.013413341)
        (41, 1.033356126)
        (43, 1.013705673)
        (47, 1.026455047)
        (53, 1.0067456)
        (59, 1.0080976)
        (61, 1.0106236)
        (67, 0.9790509)
        (71, 1.0752453)
        (73, 0.9954572)
        (79, 1.0160348)
        (83, 0.9995939)
        (89, 0.9679462)
        (97, 1.0227583)
        (101, 0.9272709)
        (103, 1.0619609)
        (107, 0.9983528)
        (109, 1.0300827)
        (113, 0.9436291)
        (127, 1.0451592)
        (131, 0.9695965)
        (137, 1.0191567)
        (139, 0.9873031)
        (149, 0.9635234)
        (151, 0.9188803)
        (157, 0.9933547)
        (163, 0.9936806)
        (167, 1.0095818)
        (173, 0.9667067)
        (179, 0.9195767)
        (181, 0.9205841)
        (191, 1.0105237)
        (193, 0.856727)
        (197, 1.0323194)
        (199, 0.9425038)
        (211, 0.979462)
        (223, 0.9796613)
        (227, 0.9842039)
        (229, 0.9911349)
        (233, 0.7817616)
        (239, 1.0795869)
        (241, 0.9409122)
        (251, 1.1277179)
        (257, 1.0105497)
        (263, 1.1689561)
        (269, 1.1794843)
        (271, 0.8666038)
        (277, 0.920748)
        (281, 0.8843632)
        (283, 1.1052282)
        (293, 1.030188)
        (307, 0.8199663)
        (311, 1.0058984)
        (313, 0.8425334)
        (317, 1.0149389)
        (331, 1.0189173)
        (337, 0.9426227)
        (347, 1.0475583)
        (349, 0.9135075)
        (353, 0.8598021)
        (359, 0.8892789)
        (367, 0.7677273)
        (373, 0.9182514)
        (379, 0.9767588)
        (383, 0.9534785)
        (389, 0.907926)
        (397, 0.9614149)
        (401, 0.8844055)
        (409, 0.836405)
        (419, 0.8953611)
        (421, 1.0988942)
        (431, 1.2445987)
        (433, 0.8999472)
        (439, 1.0792376)
        (443, 1.1578691)
        (449, 0.8265641)
        (457, 1.1904393)
        (461, 0.9988487)
        (463, 0.8360391)
        (467, 1.0032094)
        (479, 1.0783727)
        (487, 0.948676)
        (491, 1.1089726)
        (499, 0.7719031)
        (503, 0.8855315)
        (509, 1.0881402)
        (521, 0.814323)
        (523, 0.6564696)
        (541, 1.1414559)
        (547, 0.8078643)
        (557, 1.0238217)
        (563, 0.6656349)
        (569, 1.1331635)
        (571, 0.9455189)
        (577, 0.9654941)
        (587, 1.2404484)
        (593, 0.9142874)
        (599, 1.2199234)
        (601, 0.9391226)
        (607, 0.9948123)
        (613, 0.6763842)
        (617, 1.3324115)
        (619, 0.6513737)
        (631, 1.194483)
        (641, 0.7806739)
        (643, 0.9922776)
        (647, 0.7953571)
        (653, 1.0660225)
        (659, 0.868562)
        (661, 0.5679973)
        (673, 0.7699793)
        (677, 1.1916554)
        (683, 0.6530846)
        (691, 1.0027101)
        (701, 0.8353817)
        (709, 1.3572387)
        (719, 0.9305298)
        (727, 0.6870877)
        (733, 0.4835601)
        (739, 1.092242)
        (743, 1.1593029)
        (751, 0.7332013)
        (757, 0.9741833)
        (761, 0.6370331)
        (769, 0.8279054)
        (773, 0.9560464)
        (787, 1.0529273)
        (797, 0.8257717)
        (809, 0.7853772)
        (811, 1.2496699)
        (821, 0.8762533)
        (823, 1.354658)
        (827, 1.2310722)
        (829, 1.5806543)
        (839, 1.1262736)
        (853, 0.9458917)
        (857, 0.2937796)
        (859, 0.5165167)
        (863, 0.9681997)
        (877, 0.769129)
        (881, 1.1642415)
        (883, 0.779689)
        (887, 1.0227997)
        (907, 0.9049139)
        (911, 0.7469289)
        (919, 1.1823854)
        (929, 1.2082574)
        (937, 0.5267814)
        (941, 1.0625772)
        (947, 0.4484045)
        (953, 0.6357463)
        (967, 1.4026335)
        (971, 1.0371251)
        (977, 1.3363406)
        (983, 1.6426913)
        (991, 1.5713296)
        (997, 0.994009)
    };
\end{axis}
\end{tikzpicture}
\caption{Values of $M_{3,1,p,1}$ for each prime $p$ collected using the family method illustrating how the spread widens as $p$ grows.}
\label{fig:f=g-2}
\end{figure}
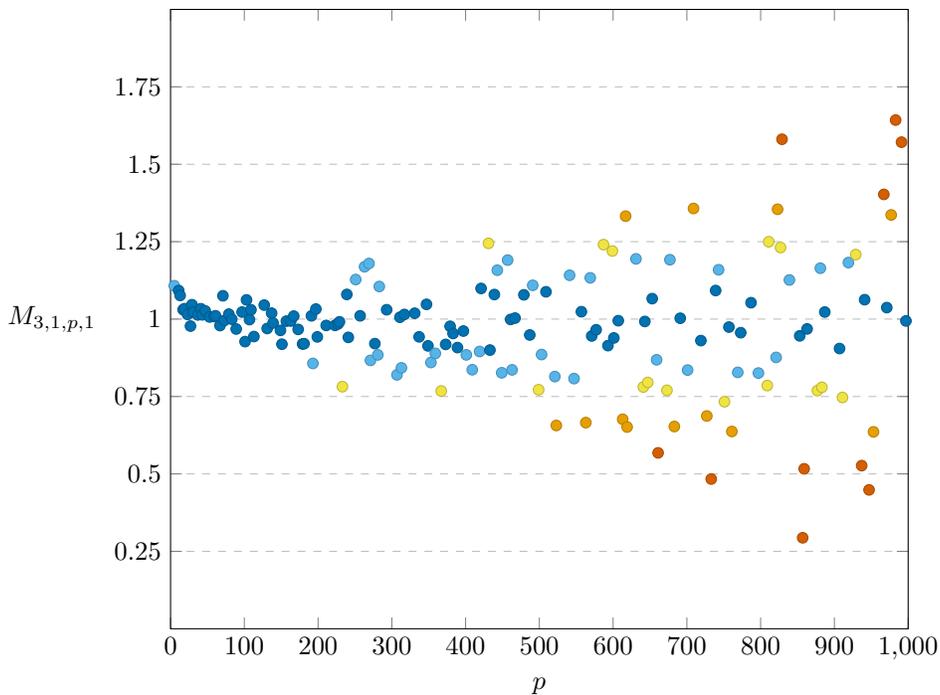

\begin{conjecture} \label{conj: f=g-2}
For $g\geq 3$ and $p\geq 3$, the moduli space $\mathcal{H}_g^{g-2}$ is geometrically irreducible. 
\end{conjecture}

Since curves of even smaller $p$-rank are very rare, it is difficult to obtain reliable heuristics for the number of components of $\mathcal{H}_g^f$ when $f < g-2$. For $3 \leq p < 100$, our data indicates that $\mathcal{H}_4^1 \otimes \Spec(\mathbb{F}_{p^r})$ and $\mathcal{H}_5^2 \otimes \Spec (\mathbb{F}_{p^r})$ are irreducible, though it is computationally infeasible to obtain statistically significant evidence. The following question aims to generalize both known theoretical results and our computational heuristics.

\begin{question} \label{q:1<=f<=g}
For $1\leq f \leq g$ and $p \geq 3$, is $\mathcal{H}_g^f$ always geometrically irreducible?   
\end{question}

The answer to this question is known to be affirmative when $f = g$ or when $(f,g)=(1,2)$. The heuristics we provide in this paper form strong evidence that the answer is also affirmative when $g-f$ is small. It seems the remaining cases are currently out of reach, both theoretically and computationally.

\subsection{$p$-Rank 0} We conclude with a brief discussion about the $p$-rank 0 locus.  Recall, in the $g=1$ case, the Deuring-Eichler mass formula gives the exact number of points of $\mathcal{H}_1^0$, which is roughly $\frac{p-1}{12}$. That is, the number of points grows with $p$.  Moreover, in the $g = 2$ case, \cite{IKO86}*{Remarks 2.16 and 2.17} show that the number of geometric components of $\mathcal{H}_2^0$ goes to infinity as $p$ increases.  \cref{fig:g=3f=0} illustrates our computational findings, showing the values of $M_{3,0,p,1}$ for each prime $p$, collected using both sampling methods.  The trend of the data indicates that the number components of $\mathcal{H}_3^0 \otimes \mathrm{Spec}(\overline{\mathbb{F}}_p)$ is also growing with $p$. 

\begin{figure}[b]
\begin{tikzpicture}
\begin{axis}[
    width={0.9\textwidth},
    xlabel={$p$},
    xmin=0, xmax=150,
    ymin=0, ymax=7,
    ymajorgrids=true,
    grid style=dashed,
    ytick = {1, 2, 3, 4, 5, 6}, 
    point meta min=0,
    point meta max=50,
    point meta=y,
    ylabel={$M_{3, 0, p, 1}$},
    ylabel style={rotate=-90},
    legend pos=south east,
    legend cell align=left,
]
\addplot+[
    colormap name=onlyorange,
    colormap access=const, 
    scatter,
    mark=*,
    only marks,
    forget plot,
    ]
    coordinates {
        (3, 1.104089219)
        (5, 1.114013522)
        (7, 1.499213241)
        (11, 1.631047096)
        (13, 1.441130168)
        (17, 1.937906552)
        (19, 2.171254978)
        (23, 2.633952461)
        (29, 2.663891076)
        (31, 2.863822166)
        (37, 3.043794041)
        (41, 2.931627951)
        (43, 3.090809802)
        (47, 3.896831764)
    };
\addplot+[
    colormap name=onlyblue,
    colormap access=const, 
    scatter,
    mark=*,
    only marks,
    forget plot,
    ]
    coordinates{
        (5, 1.107702795)
        (11, 1.044101054)
        (13, 1.079329296)
        (17, 1.579766954)
        (19, 1.45012712)
        (23, 1.81809873)
        (27, 1.030454212)
        (29, 2.162509711)
        (31, 1.905451273)
        (37, 2.188319931)
        (41, 2.276100551)
        (43, 2.173700947)
        (47, 2.669441633)
        (53, 2.9328769)
        (59, 3.5735946)
        (61, 2.7918663)
        (67, 3.308393)
        (71, 3.7580655)
        (73, 3.501153)
        (79, 2.9089301)
        (83, 4.574296)
        (89, 4.0183233)
        (97, 3.1943555)
        (101, 3.9151438)
        (103, 4.1523626)
        (107, 5.6351978)
        (109, 5.6981276)
        (113, 5.4830086)
        (127, 5.5306341)
        (131, 5.8450366)
        (137, 6.9426531)
        (139, 4.0284285)
        (149, 6.2851031)
    };
    
\addlegendimage{only marks, mark=*, color=darkblue, draw=outlineblue}
\addlegendentry{\ Family Method}

\addlegendimage{only marks, mark=*, color=darkorange, draw=outlineorange}
\addlegendentry{\ Galois Type Method}
    
\end{axis}
\end{tikzpicture}
\caption{Values of $M_{3,0,p,1}$ for each prime $p$, collected using both sampling methods.}
\label{fig:g=3f=0}
\end{figure}
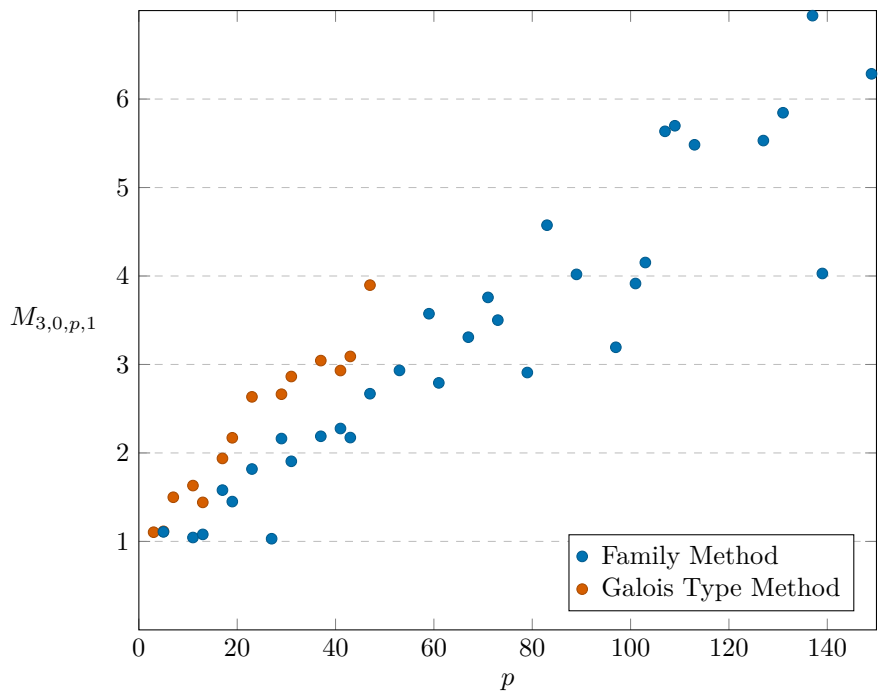

The data appears to be growing with $p$, but curiously, the values of $M_{3, 0, p, 1}$ are not concentrated around integers.  We note that a similar phenomenon occurs when running experiments on genus 2 curves, possible thanks to \cite{How25}.  That is, we once again see non-integral values when $r = 1$ for those curves.  However, since there are fewer curves than in higher genera, it is possible to enumerate them over $\mathbb{F}_{p^2}$ for small primes $p$.  What results are values of $M_{2, 0, p, 2}$ that are much closer to integers.  We have no explanation as to why this phenomenon occurs when $r = 1$, but given that we see similar behavior in smaller genera, we feel confident in the validity of our techniques. In summary, our data warrants the following conjecture.

\begin{conjecture} \label{conj:g=3 f=0}
The number of geometric components of $\mathcal{H}_3^0$ goes to infinity as $p$ grows.
\end{conjecture}

As previously noted, this conjecture generalizes the theoretically known behavior for $\mathcal{H}_1^0$ and $\mathcal{H}_2^0$. An interesting question to ask is whether this phenomenon persists for $g\geq 4$; however, the answer to this question currently seems out of reach, both theoretically and computationally.

\begin{question} \label{q:f=0}
For $g\geq 1$, does the number of geometric components of $\mathcal{H}_g^0$ go to infinity as $p$ grows?
\end{question}

\bibliographystyle{alpha} 

\end{document}